\documentclass[11pt]{article}
\usepackage{amsmath}
\usepackage{amsfonts}
\usepackage{amssymb}
\usepackage{graphicx}

\newtheorem{theorem}{Theorem}

\newtheorem{definition}[theorem]{Definition}
\newtheorem{example}[theorem]{Example}

\newtheorem{proposition}[theorem]{Proposition}
\newtheorem{remark}[theorem]{Remark}

\newenvironment{proof}[1][Proof]{\noindent\textbf{#1.} }{\ \rule{0.5em}{0.5em}}
\begin{document}

\title{Asymptotics with respect to the spectral parameter and Neumann series
of Bessel functions for solutions of the one-dimensional Schr\"{o}dinger
equation}
\author{Vladislav V. Kravchenko and Sergii M. Torba \\
{\small Departamento de Matem\'{a}ticas, CINVESTAV del IPN, Unidad Quer\'{e}%
taro, }\\
{\small Libramiento Norponiente No. 2000, Fracc. Real de Juriquilla, Quer%
\'{e}taro, Qro. C.P. 76230 MEXICO}\\
{\small e-mail: vkravchenko@math.cinvestav.edu.mx,
storba@math.cinvestav.edu.mx \thanks{%
Research was supported by CONACYT, Mexico via the project 222478.}}}
\maketitle

\begin{abstract}
A representation for a solution $u\left( \omega ,x\right) $ of the equation $%
-u^{\prime \prime }+q(x)u=\omega ^{2}u$, satisfying the initial conditions $%
u\left( \omega ,0\right) =1$, $u^{\prime }\left( \omega ,0\right) =i\omega $
is derived in the form
\begin{equation}\label{u abstr}
u\left( \omega ,x\right) =e^{i\omega x}\left( 1+\frac{u_{1}(x)}{\omega }+%
\frac{u_{2}(x)}{\omega ^{2}}\right) +\frac{e^{-i\omega x}u_{3}(x)}{\omega
^{2}}-\frac{1}{\omega ^{2}}\sum_{n=0}^{\infty }i^{n}\alpha
_{n}(x)j_{n}\left( \omega x\right),
\end{equation}
where $u_{m}(x)$, $m=1,2,3$ are given in a closed form, $j_{n}$ stands for a
spherical Bessel function of order $n$ and the coefficients $\alpha _{n}$
are calculated by a recurrent integration procedure. The following estimate
is proved $\left\vert u\left( \omega ,x\right) -u_{N}\left( \omega ,x\right)
\right\vert \leq \frac{1}{\left\vert \omega \right\vert ^{2}}\varepsilon
_{N}(x)\sqrt{\frac{\sinh \left( 2\mathop{\rm Im}\omega \,x\right) }{\mathop{\rm Im}%
\omega }}$ for any $\omega \in \mathbb{C}\backslash \left\{ 0\right\} $,
where $u_{N}\left( \omega ,x\right) $ is an approximate solution given by truncating the series in \eqref{u abstr} and $\varepsilon _{N}(x)$ is a nonnegative function
tending to zero for all $x$ belonging to a finite interval of interest. In
particular, for $\omega \in \mathbb{R}\backslash \left\{ 0\right\} $ the
estimate has the form $\left\vert u\left( \omega ,x\right) -u_{N}\left(
\omega ,x\right) \right\vert \leq \frac{1}{\left\vert \omega \right\vert ^{2}%
}\varepsilon _{N}(x)$. A numerical illustration of application of the new
representation for computing the solution $u\left( \omega ,x\right) $ on
large sets of values of the spectral parameter $\omega $ with an accuracy
nondeteriorating (and even improving) when $\omega \rightarrow \pm \infty $ is given.
\end{abstract}

\section{Introduction}

The equation
\begin{equation}
-u^{\prime \prime }+qu=\omega ^{2}u  \label{SchrodIntro}
\end{equation}%
is considered on a finite interval $\left(0,b\right) $, $\omega \in
\mathbb{C}$ with $q\in C^{1}\left[0,b\right] $ being a complex valued
function. The asymptotics with respect to the spectral parameter $\omega $
of solutions of (\ref{SchrodIntro}) is, of course, a well studied topic
exposed in a number of classical books such as \cite{Fedoryuk}, \cite%
{Marchenko}, \cite{Naimark} and \cite{Olver}. A typical result establishes
the existence for $q\in W_{2}^{n}\left[ -b,b\right] $ of a solution in the
form%
\begin{equation*}
y\left( \omega ,x\right) =e^{i\omega x}\left( u_{0}(x)+\frac{u_{1}(x)}{%
2i\omega }+\ldots +\frac{u_{n}(x)}{\left( 2i\omega \right) ^{n}}+\frac{%
u_{n+1}(\omega ,x)}{\left( 2i\omega \right) ^{n+1}}\right)
\end{equation*}%
with some additional information on $u_{n+1}(\omega ,x)$ and formulas for
computing $u_{0},\ldots ,u_{n}$ (see, e.g., \cite[Section 1.4]{Marchenko}).
However, in such constructions the solutions $y\left( \omega ,x\right) $ are
considered which are not necessarily entire functions with respect to $%
\omega $. Moreover, it is often difficult to derive the initial conditions
fulfilled by $y\left( \omega ,x\right) $, that is to identify such
solutions. Whenever such identification is possible, it represents an
additional useful result (as, e.g., in \cite[p. 35]{Fedoryuk}).

In the present work we are interested in the solution $u\left( \omega
,x\right) $ of (\ref{SchrodIntro}) satisfying the initial conditions
\begin{equation*}
u\left( \omega ,0\right) =1,\quad u^{\prime }\left( \omega ,0\right)
=i\omega .
\end{equation*}%
It is entire with respect to $\omega $ and admits the representation
\begin{equation}
u\left( \omega ,x\right) =e^{i\omega x}+\int_{-x}^{x}K(x,t)e^{i\omega t}dt
\label{u=Intro}
\end{equation}%
where $K$ is known as a transmutation (or transformation) kernel \cite[%
Chapter I]{Marchenko}.

In spite of a vast bibliography dedicated to the transmutation operators
(see, e.g., \cite{BegehrGilbert}, \cite{Carroll}, \cite{LevitanInverse},
\cite{Marchenko}, \cite{Sitnik}, \cite{Trimeche}), only recently, in \cite%
{KNT} a general representation for the kernel $K$ was derived in the form of
a Fourier-Legendre series with respect to the variable $t$ and explicit
formulas requiring a recurrent integration for its coefficients as functions
of the variable $x$. More precisely, the kernel $K$ was constructed in \cite%
{KNT} in the form
\begin{equation}
K(x,t)=\sum_{n=0}^{\infty }\frac{\beta _{n}(x)}{x}P_{n}\left( \frac{t%
}{x}\right)  \label{KIntro}
\end{equation}%
where $P_{n}$ stands for a Legendre polynomial of order $n$ and $\left\{
\beta _{n}\right\} $ are the coefficients computed following a recurrent
integration procedure.

Substitution of (\ref{KIntro}) into (\ref{u=Intro}) leads \cite{KNT} to an
interesting representation of the solution $u\left( \omega ,x\right) $ in
the form of a Neumann series of Bessel functions (NSBF) (we refer to \cite%
{Watson} and \cite{Wilkins} for more information on this kind of series),
\begin{equation}
u\left( \omega ,x\right) =e^{i\omega x}+\sum_{n=0}^{\infty }i^{n}\beta
_{n}(x)j_{n}\left( \omega x\right)   \label{NSBFIntro}
\end{equation}%
where $j_{n}$ stands for a spherical Bessel function of order $n$.

An important feature of the representation (\ref{NSBFIntro}) reveals itself
when considering an approximate solution $u_{N}\left( \omega ,x\right)
=e^{i\omega x}+\sum_{n=0}^{N}i^{n}\beta _{n}(x)j_{n}\left( \omega x\right) $%
. Let, for simplicity, $\omega \in \mathbb{R}$. Then there exists a
nonnegative function $\varepsilon _{N}(x)$ tending to zero for all $x\in %
\left[0,b\right] $ and such that
\begin{equation*}
\left\vert u\left( \omega ,x\right) -u_{N}\left( \omega ,x\right)
\right\vert \leq \varepsilon _{N}(x)  
\end{equation*}%
for all $\omega \in \mathbb{R}$. That is, the approximate solution $%
u_{N}\left( \omega ,x\right) $ approximates the exact $u\left( \omega
,x\right) $ equally well for small and for large values of $\omega $. This uniformness of approximation is preserved for complex $\omega $
meanwhile $\omega $ belongs to some strip $\left\vert \mathop{\rm Im}\omega
\right\vert <C$ in the complex plane.

This unique feature of the representation (\ref{NSBFIntro}) is a direct
consequence of the fact that it was obtained from the transmutation operator
(\ref{u=Intro}), and as was shown in \cite{KNT}, it allows one to compute in
no time the solution $u\left( \omega ,x\right) $ on a large set of values of
the spectral parameter $\omega $ with the same accuracy.

A natural question then arises whether a representation of $u\left( \omega
,x\right) $ can be obtained that would admit an estimate of the form%
\begin{equation}
\left\vert u\left( \omega ,x\right) -u_{N}\left( \omega ,x\right)
\right\vert \leq \frac{\varepsilon _{N}(x)}{\left\vert \omega \right\vert
^{k}}  \label{estimIntrok}
\end{equation}%
for all $\omega \in \mathbb{R}\backslash \left\{ 0\right\} $ and for some $k>0$, that is an
estimate even improving for large values of $\omega $.

In the present paper we show that this result is indeed possible, it is
based again on the use of the transmutation operator. We show how it can be
obtained for an arbitrary $k$ but for the sake of simplicity restrict our
consideration to $k=2$. The starting point of this work is a suggestion from
\cite[p. 51]{Marchenko} which nonetheless is followed by the sentence:
\textquotedblleft It is however difficult to find the explicit form of the
coefficients and the remainder of this expansion.\textquotedblright\ We
derive a closed form of the coefficients mentioned and construct an NSBF
representation for the remainder. The coefficients of the NSBF are
calculated by a recurrent integration procedure.

The estimate (\ref{estimIntrok}) is illustrated by a numerical example.

\section{Preliminaries: the transmutation operator, the SPPS and NSBF
representations for the solutions of the one-dimensional Schr\"{o}dinger
equation}

Let $q\in C^{1}\left[ 0,b\right] $. Consider the equation
\begin{equation}
-u^{\prime\prime}+qu=\omega^{2}u\quad\text{on }\left( 0,b\right)
,\quad\omega\in\mathbb{C},  \label{Schrod}
\end{equation}
and its solution $u\left( \omega,x\right) $ satisfying the initial
conditions
\begin{equation*}
u\left( \omega,0\right) =1,\quad u^{\prime}\left( \omega,0\right) =i\omega.
\end{equation*}
It is well known that there exists a function $K(x,t)$ defined on the domain
$0\leq\left\vert t\right\vert \leq x \leq b$ and twice
continuously differentiable with respect to each of the arguments (see,
e.g., \cite[Chapter 1]{Marchenko}) such that
\begin{equation}
u\left( \omega,x\right) =e^{i\omega x}+\int_{-x}^{x}K(x,t)e^{i\omega
t}dt\qquad\text{for all }\omega\in\mathbb{C},  \label{U=Te}
\end{equation}
and
\begin{equation}
K(x,x)=\frac{Q(x)}{2}\qquad\text{and}\qquad K(x,-x)=0  \label{K(x,x)}
\end{equation}
for all $x\in\left[ 0,b\right] $, where $Q(x):=\int_{0}^{x}q(s)ds$.

Note that from the integral equation for the function $K$ \cite[Chapter 1]
{Marchenko}
\begin{equation}
\label{IntEq for K}K(x,t) = \frac{1}2Q\left(\frac{x+t}2\right) + \int_{0}^{\frac{x+t}2} \int_{0}^{\frac{x-t}2}q(\alpha
+\beta)K(\alpha+\beta, \alpha-\beta)\,d\beta\,d\alpha
\end{equation}
one easily obtains \cite{KT2016Boletin} the equalities
\begin{equation}
K_{2}(x,x):=\frac{q(x)}{4}-\frac{Q^{2}(x)}{8}\qquad\text{and}\qquad
K_{2}(x,-x)=\frac{q(0)}{4}.  \label{K2(x,x)}
\end{equation}
Here and below by $K_{2}(x,t)$ we denote the partial derivative with respect
to the second variable.

Throughout the paper we suppose that $f_{0}$ is a solution of the equation
\begin{equation}
f^{\prime\prime}-qf=0  \label{SchrHom}
\end{equation}
satisfying the initial conditions
\begin{equation*}
f_{0}(0)=1,\quad f_{0}^{\prime}(0)=0.
\end{equation*}

Consider two sequences of recursive integrals (see \cite{KrCV08}, \cite%
{KrPorter2010})
\begin{equation}
X^{(0)}(x)\equiv1,\qquad X^{(n)}(x)=n\int_{0}^{x}X^{(n-1)}(s)\left(
f_{0}^{2}(s)\right) ^{(-1)^{n}}\,\mathrm{d}s,\qquad n=1,2,\ldots  \label{Xn}
\end{equation}
and
\begin{equation}
\widetilde{X}^{(0)}\equiv1,\qquad\widetilde{X}^{(n)}(x)=n\int_{0}^{x}%
\widetilde{X}^{(n-1)}(s)\left( f_{0}^{2}(s)\right) ^{(-1)^{n-1}}\,\mathrm{d}%
s,\qquad n=1,2,\ldots.  \label{Xtilde}
\end{equation}

\begin{definition}
\label{Def Formal powers phik and psik}The family of functions $\left\{
\varphi_{k}\right\} _{k=0}^{\infty}$ constructed according to the rule
\begin{equation}
\varphi_{k}(x)=%
\begin{cases}
f_{0}(x)X^{(k)}(x), & k\text{\ odd}, \\
f_{0}(x)\widetilde{X}^{(k)}(x), & k\text{\ even}%
\end{cases}
\label{phik}
\end{equation}
is called the system of formal powers associated with $f_{0}$.
\end{definition}

\begin{remark}
\label{Rem On particular solutions}If $f_{0}$ has zeros some of the
recurrent integrals \eqref{Xn} or \eqref{Xtilde} may not exist, although
even in that case the formal powers \eqref{phik} are well defined. It is
convenient to construct them in the following way. Take a nonvanishing
solution $f$ of \eqref{SchrHom} such that $f(0)=1$. Such solution always exists, see \cite[Remark 5]{KrPorter2010} or \cite{Camporesi et al 2011}, and for real-valued potential $q$ can be easily constructed explicitly. Indeed, one can take $f=f_{0}+if_{1}$ where $f_{1}$ is a solution of \eqref{SchrHom} satisfying $f_{1}(0)=0$, $f_{1}^{\prime }(0)=1$.
Then (see \cite[Proposition 4.7]{KT Birkhauser 2013})
\begin{equation*}
\varphi _{k}=%
\begin{cases}
\Phi _{k}, & k\text{\ odd,} \\
\Phi _{k}-\frac{f^{\prime }(0)}{k+1}\Phi _{k+1}, & k\text{\ even,}%
\end{cases}%
\end{equation*}%
where $\Phi _{k}$ are formal powers associated with $f$.
\end{remark}

Let us recall two series representations of the solution $u(\omega,x)$ which
are used in the present paper.

\begin{theorem}[Spectral Parameter Power Series, \cite{KrCV08}]
The solution $%
u(\omega,x)$ has the form
\begin{equation*}
u(\omega,x)=\sum_{n=0}^{\infty}\frac{\left( i\omega\right) ^{n}\varphi
_{n}(x)}{n!}.  
\end{equation*}
The series converges uniformly with respect to $x$ on $[0,b]$ and uniformly
with respect to $\omega$ on any compact subset of the complex plane of the
variable $\omega$.
\end{theorem}

\begin{theorem}[Neumann series of Bessel functions, \cite{KNT}] The solution $u(\omega,x)$
admits the representation
\begin{equation}
u\left( \omega,x\right) =e^{i\omega x}+\sum_{n=0}^{\infty}i^{n}\beta
_{n}(x)j_{n}\left( \omega x\right)  \label{uNSBF}
\end{equation}
where $j_{n}$ stands for the spherical Bessel function of order $n$, the
series converges uniformly with respect to $x$ on $[0,b]$ and converges
uniformly with respect to $\omega$ on any compact subset of the complex
plane of the variable $\omega$, the coefficients $\beta_{n}$ have the form
\begin{equation*}
\beta_{n}(x)=\frac{2n+1}{2}\left( \sum_{k=0}^{n}\frac{l_{k,n}\varphi_{k}(x)}{%
x^{k}}-1\right)  
\end{equation*}
where $l_{k,n}$ is the coefficient of $x^{k}$ in the Legendre polynomial of
order $n$, that is $P_{n}(x)=\sum_{k=0}^{n}l_{k,n}x^{k}$.
\end{theorem}

\section{A representation for the solutions of the one-dimensional Schr\"{o}%
dinger equation}

\begin{proposition}
The solution $u\left( \omega,x\right) $ admits the representation%
\begin{equation}
\begin{split}
u\left( \omega,x\right) &=e^{i\omega x}\left( 1+\frac{Q(x)}{2i\omega}+
\frac{1}{\omega^{2}}\left( \frac{q(x)}{4}-\frac{Q^{2}(x)}{8}\right) \right)\\
&\quad -\frac{q(0)}{4}\frac{e^{-i\omega x}}{\omega^{2}}-\frac{1}{\omega^{2}}\int
_{-x}^{x}K_{22}(x,t)e^{i\omega t}dt\qquad\text{for all }\omega\in\mathbb{C}.
\end{split}
\label{u}
\end{equation}
\end{proposition}

\begin{proof}
Multiplication of (\ref{U=Te}) by $i\omega$ and integration by parts with
the aid of (\ref{K(x,x)}) gives us the equality
\begin{equation}\label{K2}
\begin{split}
i\omega u\left( \omega,x\right) & =i\omega e^{i\omega x}+\int_{-x}^{x}K(x,t)%
\frac{de^{i\omega t}}{dt}dt \\
& =i\omega e^{i\omega x}+\frac{Q(x)}{2}e^{i\omega
x}-\int_{-x}^{x}K_{2}(x,t)e^{i\omega t}dt.
\end{split}
\end{equation}
Consider
\begin{align*}
i\omega\int_{-x}^{x}K_{2}(x,t)e^{i\omega t}dt & =\int_{-x}^{x}K_{2}(x,t)%
\frac{de^{i\omega t}}{dt}dt \\
& =K_{2}(x,x)e^{i\omega x}-K_{2}(x,-x)e^{-i\omega
x}-\int_{-x}^{x}K_{22}(x,t)e^{i\omega t}dt.
\end{align*}
Using (\ref{K2(x,x)}) we obtain
\begin{equation*}
\int_{-x}^{x}K_{2}(x,t)e^{i\omega t}dt=\frac{1}{i\omega}\left( e^{i\omega
x}\left( \frac{q(x)}{4}-\frac{Q^{2}(x)}{8}\right) -\frac{q(0)e^{-i\omega x}}{%
4}-\int_{-x}^{x}K_{22}(x,t)e^{i\omega t}dt\right) .
\end{equation*}
Substitution of this expression into (\ref{K2}) gives us (\ref{u}).
\end{proof}

\begin{remark}
By the Riemann-Lebesgue lemma on the decrease at infinity of the Fourier
transform of an $L_{1}(-\infty ,\infty )$-function, the integral $%
\int_{-x}^{x}K_{22}(x,t)e^{i\omega t}dt$ tends to zero when $\omega
\rightarrow \pm \infty $, and thus from \eqref{u} the asymptotic equality
follows
\begin{equation*}
u\left( \omega ,x\right) =e^{i\omega x}\left( 1+\frac{Q(x)}{2i\omega }+\frac{%
1}{\omega ^{2}}\left( \frac{q(x)}{4}-\frac{Q^{2}(x)}{8}\right) \right) -%
\frac{q(0)}{4}\frac{e^{-i\omega x}}{\omega ^{2}}+o\left( \frac{1}{\omega ^{2}%
}\right)
\end{equation*}%
when $\omega \rightarrow \pm \infty $.
\end{remark}

\begin{remark}Repeating this integration by parts procedure one can obtain representations involving higher order terms of $\frac 1{\omega}$ (and hence estimates \eqref{estimIntrok} with $k>2$). Explicit expressions for the terms $K_{2\ldots 2}(x,x)$ and $K_{2\ldots 2}(x,-x)$ resulting from these integrations by parts can be derived repeatedly differentiating the integral equation \eqref{IntEq for K} with respect to $t$.
\end{remark}

\section{Fourier-Legendre series expansion of the kernel $K_{22}$}

The aim of this section is to derive a representation of $K_{22}$ in the form%
\begin{equation}
K_{22}(x,t)=\sum_{n=0}^{\infty}\frac{\alpha_{n}(x)}{x}P_{n}\left( \frac{t}{x}%
\right)  \label{F-L kernel}
\end{equation}
where $P_{n}$ stands for the Legendre polynomial of order $n$, and $\alpha
_{n}$ are to be found. Note that for every $x$ fixed the function $%
K_{22}(x,\cdot)$ is continuous and hence admits a Fourier-Legendre
representation of the form (\ref{F-L kernel}) which is convergent in the
sense of $L_{2}$-norm and hence
\begin{equation}
\varepsilon_{N}(x):=\left\Vert K_{22}(x,\cdot)-K_{22,N}(x,\cdot)\right\Vert
_{L_{2}\left( -x,x\right) }\rightarrow 0, \quad N\rightarrow\infty
\label{epsilonN}
\end{equation}
for all $x\in\left[ 0,b\right] $ where $K_{22,N}(x,\cdot):=\sum_{n=0}^{N}%
\frac{\alpha_{n}(x)}{x}P_{n}\left( \frac{t}{x}\right) $. Some more precise decay rate estimates for the function $\varepsilon_N(x)$ may be obtained similarly to \cite[Theorem 3.3]{KNT} and \cite[Proposition A.2]{KT PQR}.

If $q\in C^{2}\left[ 0,b\right] $ and hence $K_{22}(x,\cdot)\in C^{1}\left[
-b,b\right] $, the series in (\ref{F-L kernel}) converges uniformly with
respect to $t\in\left[ -x,x\right] $ (see, e.g., \cite{Suetin}).

As a first step let us prove the following equalities. Denote
\begin{equation*}
\mathbf{k}_{n}(x):=\int_{-x}^{x}K_{22}(x,t)t^{n}dt,\quad n=0,1,2,\ldots.
\end{equation*}

\begin{proposition}
The following equalities are valid%
\begin{equation*}
\mathbf{k}_{0}(x)=\frac{q(x)}{4}-\frac{Q^{2}(x)}{8}-\frac{q(0)}{4},
\end{equation*}%
\begin{equation*}
\mathbf{k}_{1}(x)=\left( \frac{q(x)}{4}-\frac{Q^{2}(x)}{8}+\frac{q(0)}{4}%
\right) x-\frac{Q(x)}{2},
\end{equation*}%
\begin{equation*}
\mathbf{k}_{n}(x)=n\left( n-1\right) \left( \varphi_{n-2}(x)-x^{n-2}\right)
+\left( \frac{q(x)}{4}-\frac{Q^{2}(x)}{8}-\frac{(-1)^{n}q(0)}{4}\right)
x^{n}-\frac{nQ(x)x^{n-1}}{2},
\end{equation*}
$n=2,3,\ldots$.
\end{proposition}

\begin{proof}
Multiply (\ref{u}) by $\left( i\omega\right) ^{2}$ and make use of the SPPS
representation of the solution $u\left( \omega,x\right) $. Then
\begin{align*}
\sum_{n=0}^{\infty}\frac{\left( i\omega\right) ^{n+2}\varphi_{n}(x)}{n!} &
=\sum_{n=0}^{\infty}\frac{\left( i\omega\right) ^{n+2}x^{n}}{n!}+\frac {Q(x)%
}{2}\sum_{n=0}^{\infty}\frac{\left( i\omega\right) ^{n+1}x^{n}}{n!} \\
& \quad-\left( \frac{q(x)}{4}-\frac{Q^{2}(x)}{8}\right) \sum_{n=0}^{\infty}\frac{%
\left( i\omega\right) ^{n}x^{n}}{n!}+\frac{q(0)}{4}\sum_{n=0}^{\infty }\frac{%
\left( -i\omega\right) ^{n}x^{n}}{n!} \\
& \quad+\sum_{n=0}^{\infty}\frac{\left( i\omega\right) ^{n}}{n!}%
\int_{-x}^{x}K_{22}(x,t)t^{n}dt.
\end{align*}
Equating terms at equal powers of $\omega$ we obtain the required equalities.
\end{proof}

From (\ref{F-L kernel}) we immediately obtain a formula for the coefficients
\begin{equation*}
\alpha_{n}(x)  =\frac{2n+1}{2}\int_{-x}^{x}K_{22}(x,t)P_{n}\left( \frac {t}{%
x}\right) dt =\frac{2n+1}{2}\sum_{m=0}^{n}l_{m,n}\frac{\mathbf{k}_{m}(x)}{x^{m}}.
\end{equation*}
In particular, introducing the notation%
\begin{equation*}
q_{\pm}(x):=\frac{q(x)}{4}-\frac{Q^{2}(x)}{8}\pm\frac{q(0)}{4}
\end{equation*}
we can write
\begin{align}
\alpha_{0}(x)&=\frac{q_{-}(x)}{2},\qquad\alpha_{1}(x)=\frac{3}{2}\left(
q_{+}(x)-\frac{Q(x)}{2x}\right) ,\label{alpha_0} \\
\alpha_{2}(x)&=\frac{5}{2}\left(
q_{-}(x)+\frac{3}{x^{2}}\left( \varphi_{0}(x)-1\right) -\frac{3Q(x)}{2x}%
\right) ,\label{alpha 2} \\
\alpha_{3}(x)&=\frac{7}{2}\left( q_{+}(x)+\frac{15}{x^{3}}\left( \varphi
_{1}(x)-x\right) -\frac{3Q(x)}{2x}\right) .  \label{alpha_3}
\end{align}

In Section \ref{Sect Relation coeff} we derive a formula relating the
coefficients $\alpha_{n}$ with $\beta_{n}$, which together with the
equalities (\ref{alpha_0})--(\ref{alpha_3}) and a recurrent integration
procedure for computation of the coefficients $\beta_{n}$ from \cite{KNT}
leads to a convenient numerical algorithm for computing the coefficients $%
\alpha_{n}$.

We finish this section by the following observation, useful for controlling
the approximation accuracy.

\begin{remark}
Second differentiation of the integral equation for $K$ in analogy with \eqref{K2(x,x)} leads to the equalities
\begin{equation*}
K_{22}(x,x)=\frac{1}{8}\left( q^{\prime}(x)-q(x)Q(x)-\int_{0}^{x}q^{2}(s)ds+%
\frac{Q^{3}(x)}{6}\right)
\end{equation*}
and%
\begin{equation*}
K_{22}(x,-x)=\frac{1}{8}\left( q^{\prime}(0)+q(0)Q(x)\right) .
\end{equation*}
They can be used for a simple and efficient checking of approximation
accuracy. Indeed, from \eqref{F-L kernel} we have
\begin{equation*}
K_{22}(x,x)=\frac{1}{x}\sum_{n=0}^{\infty}\alpha_{n}(x)\qquad\text{and}\qquad
K_{22}(x,-x)=\frac{1}{x}\sum_{n=0}^{\infty}\left( -1\right) ^{n}\alpha
_{n}(x)
\end{equation*}
and hence the magnitudes
\begin{equation*}
\epsilon_{1}(x):=\left\vert \frac{1}{8}\left( q^{\prime}(x)-q(x)Q(x)-\int
_{0}^{x}q^{2}(s)ds+\frac{Q^{3}(x)}{6}\right) -\frac{1}{x}\sum_{n=0}^{N}%
\alpha_{n}(x)\right\vert
\end{equation*}
and
\begin{equation*}
\epsilon_{2}(x):=\left\vert \frac{1}{8}\left( q^{\prime}(0)+q(0)Q(x)\right) -%
\frac{1}{x}\sum_{n=0}^{N}\left( -1\right) ^{n}\alpha_{n}(x)\right\vert
\end{equation*}
characterize the accuracy of approximation of the kernel $K_{22}(x,t)$ by a
partial sum of the series \eqref{F-L kernel}.
\end{remark}

\section{Representation of the solution $u\left( \protect\omega,x\right) $}

\begin{theorem}
The solution $u\left( \omega,x\right) $ admits the following representation%
\begin{equation}
u\left( \omega,x\right)  =e^{i\omega x}\left( 1+\frac{Q(x)}{2i\omega }+
\frac{1}{\omega^{2}}\left( \frac{q(x)}{4}-\frac{Q^{2}(x)}{8}\right) \right) -
\frac{q(0)}{4}\frac{e^{-i\omega x}}{\omega^{2}}
 -\frac{1}{\omega^{2}}\sum_{n=0}^{\infty}i^{n}\alpha_{n}(x)j_{n}\left(
\omega x\right) ,  \label{ualpha}
\end{equation}
where the series converges for any $x$ on the segment $\left[ 0,b\right] $
and converges uniformly on any compact subset of the complex plane with
respect to $\omega$. Moreover, the following estimate is valid
\begin{equation}
\left\vert u\left( \omega,x\right) -u_{N}\left( \omega,x\right) \right\vert
\leq\frac{1}{\left\vert \omega\right\vert ^{2}}\varepsilon _{N}(x)\sqrt{%
\frac{\sinh\left( 2\mathop{\rm Im}\omega\,x\right) }{\mathop{\rm Im}\omega}}\qquad\text{for
any }\omega\in\mathbb{C}\backslash \left\{ 0\right\}
\label{estimate general}
\end{equation}
where
\begin{equation*}
u_{N}\left( \omega,x\right)  :=e^{i\omega x}\left( 1+\frac {Q(x)}{2i\omega}+%
\frac{1}{\omega^{2}}\left( \frac{q(x)}{4}-\frac{Q^{2}(x)}{8}\right) \right) -%
\frac{q(0)}{4}\frac{e^{-i\omega x}}{\omega^{2}}  -\frac{1}{\omega^{2}}\sum_{n=0}^{N}i^{n}\alpha_{n}(x)j_{n}\left( \omega
x\right)
\end{equation*}
and $\varepsilon_{N}$ is defined by \eqref{epsilonN}.

In particular, for all $\omega\in\mathbb{R}\backslash\left\{ 0\right\} $ the
estimate has the form $\left\vert u\left( \omega,x\right) -u_{N}\left(
\omega,x\right) \right\vert \leq\frac{1}{\left\vert \omega\right\vert ^{2}}%
\varepsilon_{N}(x)\sqrt{2x}$.
\end{theorem}

\begin{proof}
The representation (\ref{ualpha}) is obtained by substituting the series (%
\ref{F-L kernel}) into (\ref{u}) and evaluating the arising integrals with
the aid of formula 2.17.7 from \cite[p. 433]{Prudnikov}. To obtain the
estimate (\ref{estimate general}) consider
\begin{align*}
\left\vert u\left( \omega,x\right) -u_{N}\left( \omega,x\right) \right\vert
& =\frac{1}{\left\vert \omega\right\vert ^{2}}\left\vert \int_{-x}^{x}\left(
K_{22}(x,t)-K_{22,N}(x,t)\right) e^{i\omega t}dt\right\vert \\
& \leq\frac{1}{\left\vert \omega\right\vert ^{2}}\varepsilon_{N}(x)\,\left%
\Vert e^{i\omega t}\right\Vert _{L_{2}\left( -x,x\right) }
\end{align*}
where the Cauchy--Bunyakovsky--Schwarz inequality was applied. Calculation
of $\left\Vert e^{i\omega t}\right\Vert _{L_{2}\left( -x,x\right) }$ leads
to (\ref{estimate general}) from which the convergence with respect to $x$
follows.

The convergence of the series with respect to $\omega$ can be established
using the fact that for every $x$ the series represent a Neumann series of
Bessel functions (see, e.g., \cite{Watson} and \cite{Wilkins}) of the entire
function with respect to $\omega$,
\begin{equation*}
e^{i\omega x}\left( \omega^{2}+\frac{\omega Q(x)}{2i}+\frac{1}{2}\left(
\frac{q(x)}{4}-\frac{Q^{2}(x)}{8}\right) \right) -\frac{q(0)}{4}e^{-i\omega
x}-\omega^{2}u\left( \omega,x\right) .
\end{equation*}
Since the radius of convergence of the Neumann series coincides \cite[524--526%
]{Watson} with the radius of convergence of its associated power series
(obtained from the SPPS representation), the series in (\ref{ualpha})
converges uniformly on any compact subset of the complex plane of the
variable $\omega$.
\end{proof}

\section{A relation between the coefficients $\protect\alpha_{n}$ and $%
\protect\beta_{n}$\label{Sect Relation coeff}}

In \cite{KNT} a convenient recurrent integration procedure for computing the
coefficients $\left\{ \beta_{n}\right\} $ of the NSBF representation (\ref%
{uNSBF}) was derived. In order to make use of it for computing the
coefficients $\left\{ \alpha_{n}\right\} $ we establish a relation between
these two systems of functions.

\begin{proposition}
The following equality is valid%
\begin{equation}
\alpha_{n}=\left( 2n-1\right) \left( 2n+1\right) \left( \frac{\beta _{n-2}}{%
x^{2}}+\frac{2\alpha_{n-2}}{\left( 2n-5\right) \left( 2n-1\right) }-\frac{%
\alpha_{n-4}}{\left( 2n-7\right) \left( 2n-5\right) }\right)  \label{alpha_n}
\end{equation}
for $n=4,5,\ldots$.
\end{proposition}

\begin{proof}
Equating the expressions (\ref{uNSBF}) and (\ref{ualpha}) we obtain the
equality%
\begin{equation*}
e^{i\omega x}\left( \frac{Q(x)}{2i\omega}+\frac{1}{\omega^{2}}\left( \frac{%
q(x)}{4}-\frac{Q^{2}(x)}{8}\right) \right) -\frac{q(0)}{4}\frac{e^{-i\omega
x}}{\omega^{2}}-\frac{1}{\omega^{2}}\sum_{n=0}^{\infty}i^{n}%
\alpha_{n}(x)j_{n}\left( \omega x\right)
=\sum_{n=0}^{\infty}i^{n}\beta_{n}(x)j_{n}\left( \omega x\right) .
\end{equation*}
On the left-hand side the following identity can be used
\begin{equation*}
\frac{j_{n}(z)}{z^{2}}=\frac{j_{n-2}(z)}{\left( 2n-1\right) \left(
2n+1\right) }+\frac{2j_{n}(z)}{\left( 2n-1\right) \left( 2n+3\right) }+\frac{%
j_{n+2}(z)}{\left( 2n+1\right) \left( 2n+3\right) }.
\end{equation*}
Thus,
\begin{multline}
 e^{i\omega x}\left( \frac{Q(x)}{2i\omega}+\frac{1}{\omega^{2}}\left( \frac{%
q(x)}{4}-\frac{Q^{2}(x)}{8}\right) \right) -\frac{q(0)}{4}\frac{e^{-i\omega
x}}{\omega^{2}}   \\
 -x^{2}\sum_{n=0}^{\infty}i^{n}\alpha_{n}(x)\left( \frac{j_{n-2}(\omega x)}{%
\left( 2n-1\right) \left( 2n+1\right) }+\frac{2j_{n}(\omega x)}{\left(
2n-1\right) \left( 2n+3\right) }+\frac{j_{n+2}(\omega x)}{\left( 2n+1\right)
\left( 2n+3\right) }\right)   \\
 =\sum_{n=0}^{\infty}i^{n}\beta_{n}(x)j_{n}\left( \omega x\right) .
\label{eqty}
\end{multline}
Recall the orthogonality property of the spherical Bessel functions (see,
e.g., \cite[p. 732]{Arfken})%
\begin{equation}
\int_{-\infty}^{\infty}j_{m}(z)\,j_{n}(z)\,dz=
\begin{cases}
0, & m\neq n,\ m+n\geq0, \\
\frac{\pi}{2n+1}, & m=n.
\end{cases}\label{jmn}
\end{equation}
Notice that $\frac{e^{iz}}{z}=j_{-1}(z)+ij_{0}(z)$ and $\frac{e^{\pm iz}}{z^{2}}=-j_{-2}(z)\pm ij_{-1}(z)-j_{0}(z)\pm ij_{1}(z)$ and hence (\ref{eqty}) can be written in the form
\begin{multline*}
 \frac{xQ(x)}{2i}\bigl( j_{-1}(\omega x)+ij_{0}(\omega x)\bigr)
-x^{2}\left( \frac{q(x)}{4}-\frac{Q^{2}(x)}{8}\right) \bigl( j_{-2}(\omega
x)-ij_{-1}(\omega x)+j_{0}(\omega x)-ij_{1}(\omega x)\bigr) \\
 +\frac{q(0)x^{2}}{4}\bigl( j_{-2}(\omega x)+ij_{-1}(\omega x)+j_{0}(\omega
x)+ij_{1}(\omega x)\bigr) \\
 -x^{2}\sum_{n=0}^{\infty}i^{n}\alpha_{n}(x)\left( \frac{j_{n-2}(\omega x)}{%
\left( 2n-1\right) \left( 2n+1\right) }+\frac{2j_{n}(\omega x)}{\left(
2n-1\right) \left( 2n+3\right) }+\frac{j_{n+2}(\omega x)}{\left( 2n+1\right)
\left( 2n+3\right) }\right) \\
 =\sum_{n=0}^{\infty}i^{n}\beta_{n}(x)j_{n}\left( \omega x\right) .
\end{multline*}
Multiplication of this equality by $j_{m}(\omega x)$, $m\geq2$ and
integration with respect to $\omega$, with the aid of (\ref{jmn}), leads to the equality
\begin{equation*}
\beta_{m}(x)=x^{2}\left( \frac{\alpha_{m+2}(x)}{(2m+3)(2m+5)}-\frac {%
2\alpha_{m}(x)}{(2m-1)(2m+3)}+\frac{\alpha_{m-2}(x)}{(2m-3)(2m-1)}\right)
\end{equation*}
which gives us (\ref{alpha_n}).
\end{proof}

\section{Numerical illustration}

A numerical approach based on the NSBF representation (\ref{uNSBF}) for
solving equation (\ref{Schrod}) for a wide range of values of the spectral
parameter $\omega$ as well as related spectral problems was presented in
\cite{KNT}. It is an efficient, simple and practical method which apart from
its numerical advantages is available for an unaided programming by
researchers which does not require any sophisticated numerical technique
usually encoded in purely numerical solvers. The aim of this section is to
show that another NSBF representation (\ref{ualpha}) derived in the present
work can give even more accurate results when the number of the computed
terms is limited.

\begin{example}
\label{Example Exp} Consider the following spectral problem (the first Paine
problem, \cite{Paine}, see also \cite[Example 7.2]{KNT})
\begin{equation*}
\begin{cases}
-u^{\prime\prime}+e^{x}u=\lambda u,\qquad 0\leq x\leq\pi, \\
u(0,\lambda)=u(\pi,\lambda)=0.%
\end{cases}%
\end{equation*}

We solve this problem numerically using partial sums from \eqref{uNSBF} and \eqref{ualpha} and referring the reader to \cite{KNT} where relevant details
on numerical aspects can be found. Thus, an approximate solution of the
Sturm-Liouville problem is computed using the approximate solutions
\begin{equation}
\widetilde{u}_{N}\left( \pm \omega ,x\right) =e^{\pm i\omega
x}+\sum_{n=0}^{N}i^{n}\beta _{n}(x)j_{n}\left( \pm \omega x\right)
\label{u_N1}
\end{equation}%
and
\begin{equation}
u_{N}\left( \pm \omega ,x\right)  :=e^{\pm i\omega x}\left( 1\pm \frac{Q(x)%
}{2i\omega }+\frac{1}{\omega ^{2}}\left( \frac{q(x)}{4}-\frac{Q^{2}(x)}{8}%
\right) \right) -\frac{q(0)}{4}\frac{e^{\mp i\omega x}}{\omega ^{2}}   -\frac{1}{\omega ^{2}}\sum_{n=0}^{N+2}i^{n}\alpha _{n}(x)j_{n}\left( \pm
\omega x\right) .  \label{uN}
\end{equation}
The upper bound of summation $N+2$ in \eqref{uN} in comparison with $N$
in \eqref{u_N1} is due to the fact that $N$ computed coefficients $\beta _{n}$
are transformed into $N+2$ computed coefficients $\alpha _{n}$ (see \eqref{alpha_n}).

\begin{figure}[htb!]
\centering
\begin{tabular}{c}
\includegraphics[bb=117 330 477 510, width=5in,height=2.5in]{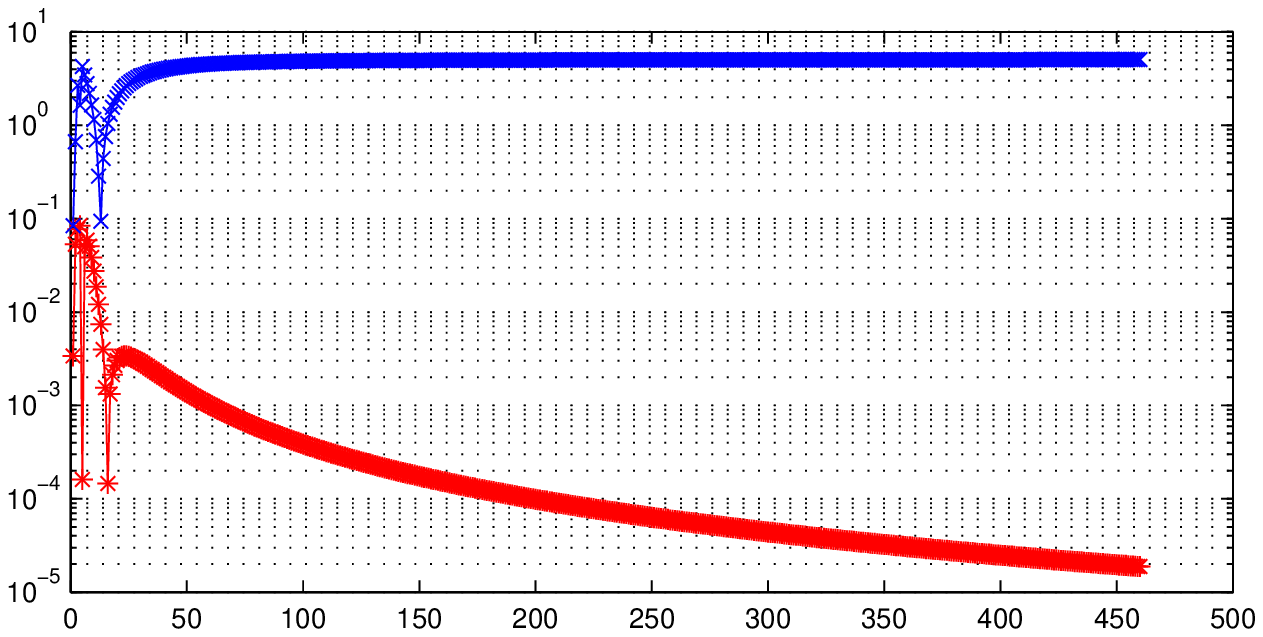}\\
(a) $N=5$\\
\includegraphics[bb=117 330 477 510, width=5in,height=2.5in]{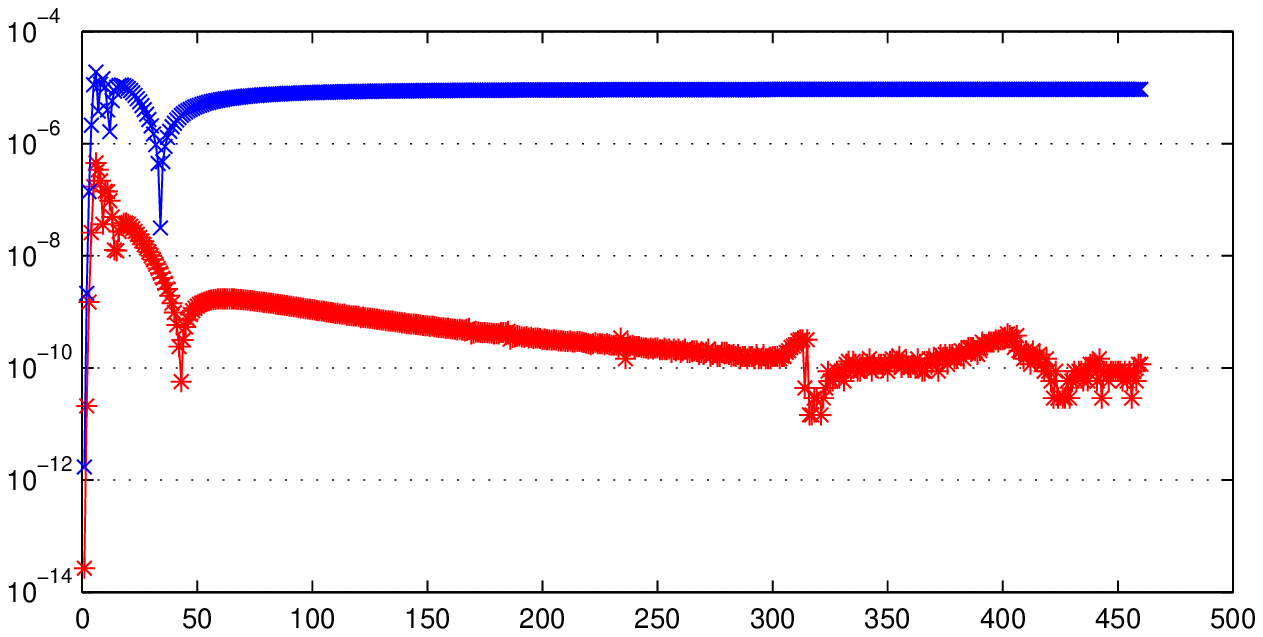}\\
(b) $N=15$\\
\includegraphics[bb=117 330 477 510, width=5in,height=2.5in]{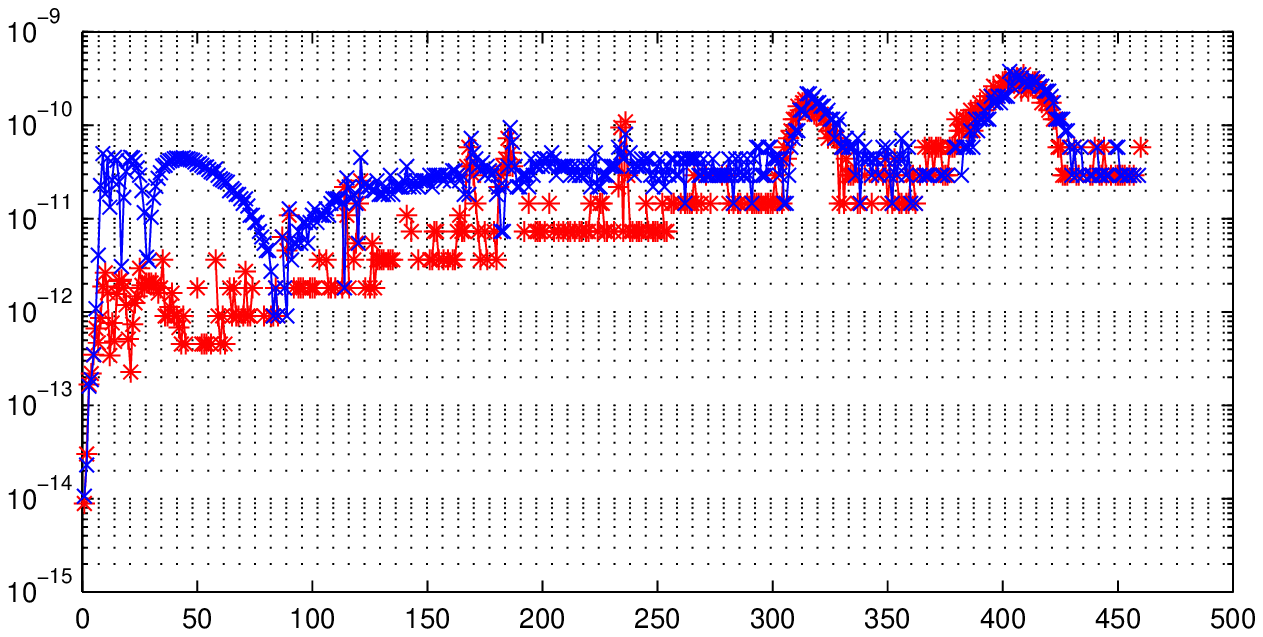}\\
(c) $N=25$
\end{tabular}
\caption{Comparison of the absolute
errors of approximation of the first 460 eigenvalues computed using partial
sums (\protect\ref{u_N1}) (in blue -$\times$- lines) and (\protect\ref{uN})
(in red -*- lines). (a) with $N=5$, (b) with $N=15$ and (c) with $N=25$.}
\label{Fig1}
\end{figure}

For relatively small $N$ the eigenvalues are approximated more accurately by
the solution obtained from \eqref{uN}. This can be appreciated on Figs.\,\ref{Fig1} (a) and (b) where the absolute error of the first 460 eigenvalues computed is
presented for $N=5$ and $N=15$ respectively. However for larger $N$ the
difference in the accuracy practically disappears, as is illustrated by Fig.\,\ref{Fig1} (c) where the same comparison is made for $N=25$. Both approximations deliver
excellent numerical results in fractions of a second on a usual computer.

We emphasize that the eigenvalues of the problem behave asymptotically \cite{Levitan1950} as $\sim\left( n+\frac{Q(\pi )}{2\pi n}\right) ^{2}$, however the values computed by both our algorithms are still closer to the exact one.
For example, the exact value of $\lambda_{460}$ is $211607.047634847$ (rounded up to the presented digits), the asymptotic expression gives $211607.047660$, while the approximation based on \eqref{uN}
delivered the value $211607.047634847$. Moreover, due to the \textquotedblleft
largeness\textquotedblright\ of the higher eigenvalues the errors of approximate values computed using either of NSBF
representations are limited by machine precision (while absolute errors are of order $10^{-9}$, relative errors are less than $10^{-15}$).
\end{example}


\begin{thebibliography}{99}
\itemsep = -2pt

\bibitem{Arfken} G.~Arfken, H.~Weber, \emph{Mathematical methods for
physicists}, Elsevier Academic Press, 2005.

\bibitem{BegehrGilbert} H.~Begehr and R.~Gilbert, \emph{Transformations,
transmutations and kernel functions}, vol. 1--2, Harlow: Longman Scientific
\& Technical, 1992.

\bibitem{Camporesi et al 2011}R.~Camporesi and A.~J.~Di Scala, \emph{A
generalization of a theorem of Mammana}, Colloq. Math. 122 (2011),
no. 2, 215--223.

\bibitem{Carroll} R.~W.~Carroll, \emph{Transmutation theory and applications}, Mathematics Studies, Vol. 117, North-Holland, 1985.

\bibitem{Fedoryuk} M.~V.~Fedoryuk, \emph{Asymptotic analysis. Linear
ordinary differential equations}, Berlin: Springer-Verlag, 1993.

\bibitem{KrCV08} V.~V.~Kravchenko, \emph{A representation for solutions of
the Sturm-Liouville equation}, Complex Var. Elliptic Equ. 53
(2008), 775--789.

\bibitem{KNT} V.~V.~Kravchenko, L.~J.~Navarro and S.~M.~Torba, \emph{Representation of solutions to the one-dimensional Schr\"{o}dinger equation in terms of Neumann series of Bessel functions}, submitted, available at arXiv:1508.02738.

\bibitem{KrPorter2010} V.~V.~Kravchenko and R.~M.~Porter, \emph{Spectral
parameter power series for Sturm-Liouville problems}, Math. Methods Appl.
Sci. 33 (2010), 459--468.

\bibitem{KT Birkhauser 2013} V.~V.~Kravchenko and S.~M.~Torba, \emph{Transmutations and spectral parameter power series in eigenvalue problems},
In: Operator Theory: Advances and Applications, 228 (2013), 209--238.

\bibitem{KT2016Boletin} V.~V.~Kravchenko and S.~M.~Torba, \emph{Analytic
approximation of transmutation operators and related systems of functions},
Bol. Soc. Mat. Mex. 22 (2) (2016), 389--429.

\bibitem{KT PQR}V.~V.~Kravchenko and S.~M.~Torba, \emph{A Neumann series of Bessel functions representation for solutions of Sturm-Liouville equations}, submitted, available at arXiv:1612.08803.

\bibitem{Levitan1950} B.~M.~Levitan, \emph{Expansion in characteristic
functions of differential equations of the second order}, Moscow-Leningrad: Gosudarstv. Izdat.
Tehn.-Teor. Lit., 1950. 159 pp. (in Russian).

\bibitem{LevitanInverse} B.~M.~Levitan, \emph{Inverse Sturm-Liouville
problems}, Zeist: VSP, 1987.

\bibitem{Marchenko} V.~A.~Marchenko, \emph{Sturm-Liouville operators and
applications: revised edition}, AMS Chelsea Publishing, 2011.

\bibitem{Naimark} M.~A.~Naimark, \emph{Linear differential operators. Part
I: Elementary theory of linear differential operators}, New York: Frederick Ungar Publishing Co., 1967.

\bibitem{Olver} F.~W.~J.~Olver, \emph{Asymptotics and special functions},
Wellesley, MA: AKP Classics,  1997.

\bibitem{Paine}J.~W.~Paine, F.~R.~de Hoog and R.~S.~Anderssen, \emph{On the
correction of finite difference eigenvalue approximations for Sturm-Liouville
problems}, Computing 26 (1981), 123--139.

\bibitem{Prudnikov} A.~P.~Prudnikov, Yu.~A.~Brychkov and O.~I.~Marichev,
\emph{Integrals and series. Vol. 2. Special functions}, New York: Gordon \& Breach Science Publishers, 1986. 750 pp.

\bibitem{Sitnik} S.~M.~Sitnik, \emph{Transmutations and applications: a
survey}, arXiv:1012.3741v1, originally published in the book: \emph{Advances
in Modern Analysis and Mathematical Modeling}, Editors: Yu.~F.~Korobeinik,
A.~G.~Kusraev, Vladikavkaz: Vladikavkaz Scientific Center of the Russian
Academy of Sciences and Republic of North Ossetia--Alania, 2008, 226--293.

\bibitem{Suetin} P.~K.~Suetin, \emph{Classical orthogonal polynomials, 3rd
ed.}, (in Russian),  Moscow: Fizmatlit, 2005, 480 pp.

\bibitem{Trimeche} K.~Trimeche, \emph{Transmutation operators and
mean-periodic functions associated with differential operators}, London:
Harwood Academic Publishers, 1988.

\bibitem{Watson} G.~N.~Watson, \emph{A Treatise on the theory of Bessel
functions, 2nd ed., reprinted}, Cambridge, UK: Cambridge University Press,
1996, vi+804 pp.

\bibitem{Wilkins} J.~E.~Wilkins, \emph{Neumann series of Bessel functions},
Trans. Amer. Math. Soc. 64 (1948), 359--385.
\end{thebibliography}
\end{document}